\documentclass[12pt]{article}

\usepackage{latexsym,amsmath,amscd,amssymb,graphics,float}
\usepackage{enumerate}

\usepackage{graphicx,lscape}

\usepackage[colorlinks]{hyperref}
\usepackage{url}

\begin{document}

\newenvironment{proof}[1][Proof]{\textbf{#1.} }{\ \rule{0.5em}{0.5em}}

\newtheorem{theorem}{Theorem}[section]
\newtheorem{definition}[theorem]{Definition}
\newtheorem{lemma}[theorem]{Lemma}
\newtheorem{remark}[theorem]{Remark}
\newtheorem{proposition}[theorem]{Proposition}
\newtheorem{corollary}[theorem]{Corollary}
\newtheorem{example}[theorem]{Example}

\numberwithin{equation}{section}
\newcommand{\ep}{\varepsilon}
\newcommand{\R}{{\mathbb  R}}
\newcommand\C{{\mathbb  C}}
\newcommand\Q{{\mathbb Q}}
\newcommand\Z{{\mathbb Z}}
\newcommand{\N}{{\mathbb N}}

\newcommand{\bfi}{\bfseries\itshape}

\newsavebox{\savepar}
\newenvironment{boxit}{\begin{lrbox}{\savepar}
\begin{minipage}[b]{15.5cm}}{\end{minipage}\end{lrbox}
\fbox{\usebox{\savepar}}}

\title{{\bf Asymptotic stabilization with phase of periodic orbits of three-dimensional Hamiltonian systems}}
\author{R\u{a}zvan M. Tudoran}

\date{}
\maketitle \makeatother

\begin{abstract}
We provide a geometric method to stabilize asymptotically with phase an arbitrary fixed periodic orbit of a locally generic three-dimensional Hamiltonian dynamical system. The main advantage of this method is that one needs not know a parameterization of the orbit to be stabilized, but only the values of the Hamiltonian and a fixed Casimir (of the Poisson configuration manifold) at that orbit. The stabilization procedure is illustrated in the case of the Rikitake model of geomagnetic reversal.
\end{abstract}

\medskip

\textbf{AMS 2000}: 37J25; 37J45; 37C27; 34C25.


\section{Introduction}
\label{section:one}

The purpose of this article is to provide a method to stabilize asymptotically with phase an arbitrary fixed periodic orbit of a locally generic three-dimensional Hamiltonian dynamical system. The main advantage of this method is that one needs not know a parameterization of the orbit to be stabilized, but only the values of the Hamiltonian and a fixed Casimir (of the Poisson configuration manifold) at that orbit. Moreover, if there are many periodic orbits located on the same common level set of the Hamiltonian and the Casimir, then the same perturbation can be used in order to asymptotically stabilize all of them in the same time. The method can be applied for a large number of concrete dynamical systems coming from various sciences, which admit three-dimensional Hamiltonian realizations, e.g., Euler's equations of free rigid body dynamics (\cite{mr}, \cite{dgt}), the Rikitake system(\cite{siamt}), the R\"ossler system (\cite{tgr}), the Rabinovich system (\cite{tgrab}), etc.

Before explaining how the stabilization method works, let us clarify what is meant by a locally generic three-dimensional Hamiltonian system. Formally, a \textit{locally generic three-dimensional Hamiltonian system} is a three-dimensional Hamiltonian system restricted to an open neighborhood around a regular point of the Poisson configuration manifold. Recall that in contrast to Casimir invariants (which globally may not exist), locally, around each regular point, every Poisson manifold admits \textit{local Casimir invariants} (i.e., Casimir invariants of the restricted Poisson structure). Accordingly, a locally generic three-dimensional Hamiltonian system is a dynamical system modeled on a Poisson manifold with nontrivial Casimir invariants, each system of this type admitting, apart from the Hamiltonian, another constant of motion, namely some fixed Casimir invariant (as any two Casimir invariants are functionally dependent, we fix one of them, referred to as \textit{the} Casimir invariant). Due also to the three-dimensionality of the problem, each orbit of the system is located on a common level set of the Hamiltonian and the Casimir invariant. Note that it happens often that there are many obits located on the same common level set. If the Hamiltonian or the Casimir invariant is a proper function, then one expects the dynamical system has periodic orbits. In the case the system admits periodic orbits, the stabilization method works as follows: pick a common level set containing a periodic orbit of the dynamical system; than the method tells how to perturb the system in such a way that the periodic orbit remains a periodic orbit for the perturbed system too, the perturbed system keeps dynamically invariant the Hamiltonian (or the Casimir invariant) and moreover, the periodic orbit of the perturbed system is orbitally phase asymptotically stable with respect to perturbations along the level set of the Hamiltonian, which contains the periodic orbit (or along the level set of the Casimir invariant, which contains the periodic orbit). Moreover, using a specific linear combination of the above perturbations, one obtains a perturbed system for which the periodic orbit is orbitally phase asymptotically stable with respect to perturbations in a "full" open neighborhood of the orbit. The price to be paid is the loss of the conservative nature of the perturbed system. On the other hand, the method can be used also in order to obtain the same conclusions, except that the periodic orbit becomes unstable for the perturbed system. In all cases, the perturbed families of dynamical systems are parameterized by arbitrary strictly positive smooth real functions. This method is an improvement of the $n$-dimensional stabilization technique introduced in \cite{TudoranJGM}, for the particular case of locally generic three-dimensional Hamiltonian systems. In contrast with the general case analyzed in \cite{TudoranJGM}, where in order to apply the stabilization method, one needs to know a parameterization of the periodic orbit to be stabilized, here, the only requirement regarding the periodic orbit is the knowledge of the values of the Hamiltonian and the Casimir at the orbit. Moreover, for the class of locally generic three-dimensional Hamiltonian systems, we prove that one of the main hypotheses from \cite{TudoranJGM} is always fulfilled.

The structure of the paper is the following. In the second section, we discuss the local equivalence between the class of three-dimensional Hamiltonian dynamical systems, and that of three-dimensional completely integrable systems. The last section is the main part of this paper, here we introduce the stabilization technique and also illustrate the main result in the case of the Rikitake model of geomagnetic reversal.

\section{3D Hamiltonian systems versus 3D completely integrable systems}

In this section we recall the local equivalence between the class of 3D Hamiltonian dynamical systems modeled on Poisson manifolds, and that of 3D completely integrable systems. For more details regarding the geometry of Poisson manifolds and related topics see, e.g., \cite{zung}, \cite{mr}, \cite{gurs}.

Let us start by choosing an open set $U\subseteq \mathbb{R}^3$ and a smooth bivector field $\Pi\in\mathfrak{X}^2(U)=\Gamma(\Lambda^{2}TU)$ given locally by
$$\Pi (x,y,z)=\alpha(x,y,z)\frac{\partial}{\partial
y}\wedge\frac{\partial}{\partial
z}+\beta(x,y,z)\frac{\partial}{\partial
z}\wedge\frac{\partial}{\partial
x}+\gamma(x,y,z)\frac{\partial}{\partial
x}\wedge\frac{\partial}{\partial y}, ~ (\forall)(x,y,z)\in U,$$ 
where $\alpha,\beta,\gamma \in\mathcal{C}^{\infty}(U,\mathbb{R})$ are smooth real functions. 

Recall that the bivector field $\Pi$ is a \textit{Poisson structure} on $U$ if and only if $[\Pi,\Pi]=0$ (where $[\cdot,\cdot]$ denotes the Schouten-Nijenhuis bracket), this condition being in turn equivalent to the integrability of the smooth $1$-form $\omega=\alpha\mathrm{d}x+\beta\mathrm{d}y+\gamma\mathrm{d}z$, namely, $\mathrm{d}\omega\wedge\omega=0$. In the case when $\Pi$ is a Poisson structure on $U$, the pair $(U,\Pi)$ is a \textit{Poisson manifold}.

Let us recall now the definition of the \textit{sharp} isomorphism between the space of smooth $1-$forms, $\Omega^{1}(U)$, and the space of smooth vector fields, $\mathfrak{X}(U)$, induced by the canonical inner product on $\mathbb{R}^3$.  More precisely, to each smooth $1-$form $\omega = \alpha\mathrm{d}x+\beta\mathrm{d}y+\gamma\mathrm{d}z$, $\alpha,\beta,\gamma \in\mathcal{C}^{\infty}(U,\mathbb{R})$, we associate a smooth vector field, $\omega^{\sharp}$, given by
$$\omega^{\sharp}=\alpha\dfrac{\partial}{\partial x}+\beta\dfrac{\partial}{\partial y}+\gamma\dfrac{\partial}{\partial z}. $$

In the above notations, given a Poisson structure $\Pi$ on $U$, a three-dimensional dynamical system of the type 
$$
\dfrac{\mathrm{d}u}{\mathrm{d}t} =\widetilde{\Pi}(u)\nabla H(u), 
$$
where $H\in\mathcal{C}^{\infty}(U,\mathbb{R})$ is a smooth real function, and
$$\widetilde{\Pi}(u):=\left[ {\begin{array}{*{20}c}
   0 & \gamma(u) & -\beta(u) \\\\
   -\gamma(u) & 0 & \alpha(u) \\\\
    \beta(u)& -\alpha(u) & 0  \\
\end{array}} \right], ~ (\forall) u\in U,$$
is called \textit{Hamiltonian dynamical system}. 

The associated vector field, $u\in U\longmapsto \widetilde{\Pi}(u)\nabla H(u) \in T_{u}U\cong \mathbb{R}^{3}$, is called the \textit{Hamiltonian vector field} associated to the energy/Hamiltonian $H\in C^{\infty}(U,\mathbb{R})$, and is denoted by $X_{H}$.

Using the canonical Lie algebra isomorphism between $(\mathbb{R}^3,+,\cdot_{\mathbb{R}} ,\times)$ and $(\mathfrak{so}(3),+,\cdot_{\mathbb{R}},[\cdot,\cdot])$, the so called "hat" map, $\hat ~: \mathbb{R}^3 \rightarrow \mathfrak{so}(3)$, it follows that $$\widetilde{\Pi}(u)=-\widehat{\omega^{\sharp}(u)}, ~(\forall)u\in U.$$ Consequently, using the properties of the "hat" map we get that for every $u\in U$, 
$$X_{H}(u)=\widetilde{\Pi}(u)\nabla H(u)=-\widehat{\omega^{\sharp}(u)}\nabla H(u)= -\omega^{\sharp}(u) \times \nabla H(u)=\nabla H(u)\times\omega^{\sharp}(u).$$

Let us recall now a generic local form of Hamiltonian dynamical systems modeled on a three-dimensional Poisson manifold. In order to do that, we first recall a generic local form of a general three-dimensional smooth integrable $1-$form, $\omega$. More precisely, around each regular point (i.e., a point where $\omega$ in nonzero), $\omega$ can be written in the form $\omega:=\omega_{\nu, C}=\nu \mathrm{d}C$, for some \textit{locally defined} smooth real functions $C$ and $\nu$, with $\nu$ non-vanishing. The associated Poisson structure will be denoted by $\Pi_{\nu,C}$. As the integrable $1-$form $\omega=y\mathrm{d}x - x \mathrm{d}y$ do not admits a global representation of the type $\nu \mathrm{d}C$, it follows that not every three-dimensional Poisson structure can be written \textit{globally} as $\Pi_{\nu,C}$.

Summarizing, each three-dimensional Hamiltonian dynamical system, can be realized locally (around regular points) as a Hamiltonian system defined on a Poisson manifold of type $(U,\Pi_{\nu, C})$, where $U\subseteq\mathbb{R}^3$ is an open set, and $\nu,C\in\mathcal C^{\infty}(U,\mathbb{R})$ are smooth real functions, with $\nu$ non-vanishing. More precisely, as $(\nu  \mathrm{d}C)^{\sharp}=\nu  \nabla C$, the expression of the Hamiltonian vector field generated by a smooth function $H\in C^{\infty}(U,\mathbb{R})$, is given by
\begin{align*}
X_{H}(u)&=\nabla H(u)\times\omega_{\nu, C}^{\sharp}(u) =\nabla H(u)\times (\nu  \mathrm{d}C)^{\sharp}(u)=\nabla H(u)\times (\nu  \nabla C)(u)\\
&=\nu(u)\left( \nabla H(u)\times \nabla C(u)\right), ~(\forall)u\in U,
\end{align*}
and consequently the associated Hamiltonian dynamical system becomes
\begin{align}\label{ghs}
\dfrac{\mathrm{d}u}{\mathrm{d}t} =\nu(u)\left(\nabla H(u)\times \nabla C(u)\right).
\end{align}

A straightforward computation shows that $H$ and $C$ are both first integrals of the Hamiltonian system \eqref{ghs}. The smooth function $C$, is called \textit{Casimir} function of the Poisson structure $\Pi_{\nu,C}$, and has the dynamical property that it is a first integral of \textit{each} Hamiltonian system realized on the Poisson manifold $(U,\Pi_{\nu,C})$.

Consequently, any three-dimensional Hamiltonian system, is locally a \textit{completely integrable system}, i.e., a three-dimensional dynamical system which admits two first integrals (functionally independent at least on some open subset of the domain of definition). Conversely, we have that that any three-dimensional \textit{completely integrable} dynamical system, can be expressed \textit{at least locally} as a Hamiltonian dynamical system of type \eqref{ghs}. For a similar $n$-dimensional result, see \cite{tjgp}.

Indeed, let $\dfrac{\mathrm{d}u}{\mathrm{d}t} = X(u)$, be a three-dimensional dynamical system generated by a smooth vector field $X\in \mathfrak{X}(\mathbb{R}^3)$. This system will be \textit{completely integrable}, on an open subset $U\subseteq \mathbb{R}^3$, if there exist $I_1, I_2 \in C^{\infty}(U,\mathbb{R})$ two first integrals of $X$, functionally independent at least on some open subset of $U$. 

Hence, the smooth functions $I_1, I_2$ verify the equalities
\begin{align*}
\left(\mathcal{L}_{X} I_1 \right)(u)=0, \quad \left(\mathcal{L}_{X} I_2 \right)(u)=0, ~(\forall)u\in U,
\end{align*}
or equivalently 
\begin{align}\label{fic}
\langle X(u),\nabla I_1 (u)\rangle = 0,\quad \langle X(u),\nabla I_2 (u)\rangle = 0, ~(\forall)u\in U.
\end{align}

From relations \eqref{fic} it follows that for any $u\in U$, the vector $X(u)$ is orthogonal to both $\nabla I_1 (u)$ and $\nabla I_2 (u)$, hence $X(u)$ must be collinear with $\nabla I_1 (u) \times \nabla I_2 (u)$.

Consequently, depending on the existence and the location of the equilibrium points of the vector field $X$, there exists an open set $\Omega\subseteq U$ and a smooth function $\mu \in C^{\infty}(\Omega,\mathbb{R})$ such that $$X(u)=\mu (u)\left( \nabla I_1 (u)\times \nabla I_2 (u)\right),$$ for any $u\in \Omega$. When dealing with concrete dynamical systems, one often notes that the open set $\Omega$ is a proper subset of $U$.

Consequently, any three-dimensional completely integrable dynamical system is \textit{locally} a Hamiltonian system, defined on a Poisson manifold of the type $(\Omega,\Pi_{\nu, C})$, with $\Omega\subseteq \mathbb{R}^3$ an open subset of $\mathbb{R}^3$, $\nu=\mu$, $C=I_2|_{\Omega}$, and energy $H=I_1|_{\Omega}$.

\textbf{Summarizing, in dimension three, completely integrable systems and Hamiltonian systems are locally the same.}

Moreover, the generic integral curves of a Hamiltonian dynamical of type \eqref{ghs}, are located on the connected components of the common level surfaces corresponding to regular values of $H$ and $C$. Consequently, if any of the smooth functions $C$ or $H$ is proper, then one expects generic existence of periodic orbits of the system. Recall that a smooth function $F\in C^{\infty}(\mathbb{R}^3,\mathbb{R})$ is called \textit{proper} if the preimage of every compact set in $\mathbb{R}$ is compact in $\mathbb{R}^3$, which is equivalent to $\lim_{\|u\|\rightarrow \infty}|F(u)|=\infty$.

Note that a large number of concrete dynamical systems from various sciences admits three-dimensional Hamiltonian realizations of this type, e.g., Euler's equations from the free rigid body dynamics (\cite{mr}, \cite{dgt}), the Lotka-Volterra system(\cite{tglv}), the Rikitake system(\cite{siamt}), the R\"ossler system (\cite{tgr}), the Rabinovich system (\cite{tgrab}), etc.

\section{Geometric asymptotic stabilization of periodic orbits of three-dimensional Hamiltonian systems}

The purpose of this section is to improve, for locally generic three-dimensional Hamiltonian dynamical systems, the stabilization technique introduced in \cite{TudoranJGM}, and also to provide a new stabilization result. More precisely, we show that one of the main hypothesis required by the stabilization technique from \cite{TudoranJGM}, is always satisfied in the case of locally generic three-dimensional Hamiltonian systems. Moreover, we will show that the elimination of this hypothesis leads to a new stabilization method, which provides phase asymptotic stability of a certain periodic orbit, even in the case when we are not able to find a parameterization of the orbit (which is the case in many concrete dynamical systems). The only requirement concerning the periodic orbit to be stabilized is the geometric location, more exactly, an implicit characterization of a dynamically invariant set containing it. 

Let us start by recalling some definitions concerning the stability of the periodic orbits of a dynamical system. In order to do that, let $\dot x =X(x)$ be a dynamical system generated by a smooth vector field $X\in\mathfrak{X}(U)$, defined on an open subset $U$ of an $n-$dimensional smooth Riemannian manifold, whose canonical distance (induced by the metric) we denote by $dist$. Suppose $\Gamma =\{\gamma(t)\in U : 0\leq t\leq T \}$ is a $T-$periodic orbit of $\dot x =X(x)$. For more details regarding the stability analysis of periodic orbits see e.g., \cite{hartman}, \cite{verhulst}, \cite{moser}.

\begin{definition}[\cite{TudoranJGM}]
\begin{enumerate}
\item The periodic orbit $\Gamma$ is called \textbf{orbitally stable} if, given $\varepsilon >0$ there exists a $\delta >0$ such that $\operatorname{dist}(x(t;x_0),\Gamma)<\varepsilon $ for all $t>0$ and for all $x_0 \in U$ such that $\operatorname{dist}(x_0,\Gamma)<\delta $, where $t\mapsto x(t;x_0)$ stands for the solution of the Cauchy problem $\dot x =X(x)$, $x(0)=x_0$.
\item The periodic orbit $\Gamma$ is called \textbf{unstable} if it is not orbitally stable.
\item The periodic orbit $\Gamma$ is called \textbf{orbitally asymptotically stable} if it is orbitally stable and moreover (by choosing $\delta$ from item $(1)$ smaller if necessary), $$\operatorname{dist}(x(t;x_0),\Gamma)\rightarrow 0 ~as~ t\rightarrow \infty.$$
\item The periodic orbit $\Gamma$ is called \textbf{orbitally phase asymptotically stable}, if it is orbitally asymptotically stable and there is a $\delta >0$ such that for each $x_0 \in U$ with $\operatorname{dist}(x_0,\Gamma)<\delta $, there exists $\theta_{0}=\theta_0 (x_0)$ such that $$\lim_{t\rightarrow \infty}\operatorname{dist}(x(t;x_0),\gamma(t+\theta_{0}))=0.$$
\end{enumerate}
\end{definition}

We recall now the main theorem from \cite{TudoranJGM} which provides for the class of $n$ - dimensional completely integrable dynamical systems, a constructive perturbation method which keeps an a-priori fixed proper subset of first integrals of the system, preserves a given periodic orbit, and implements an a-priory prescribed stability behavior of the orbit.

\begin{theorem}[\cite{TudoranJGM}]\label{OST}
Let $\dot x= X(x)$ be a completely integrable dynamical system generated by a smooth vector field $X\in\mathfrak{X}(U)$ defined eventually on an open subset $U\subseteq \mathbb{R}^{n}$, and let $k,p\in\mathbb{N}$ be two natural numbers, with $k+p=n-1$, such that there exist $n-1$ first integrals $J_1,\dots,J_k, D_1,\dots, D_p\in\mathcal{C}^{\infty}(U,\mathbb{R})$, independent on an open subset $V \subseteq U$. Suppose the system $\dot x= X(x)$ admits a $T-$periodic orbit $\Gamma=\{\gamma(t)\in V : 0\leq t\leq T \}$ such that:
\begin{itemize}
\item $\Gamma\subset JD^{-1}(\{0\})$, and $0 \in \mathbb{R}^{n-1}$ is a regular value of the map $$JD=(J_1,\dots,J_k,D_1,\dots,D_p):V\subseteq \mathbb{R}^{n}\rightarrow \mathbb{R}^{n-1},$$
\item $\nabla J_1(\gamma(t)),\dots,\nabla J_k(\gamma(t)),\nabla D_1(\gamma(t)),\dots, \nabla D_p (\gamma(t)), X(\gamma(t))$ are linearly\\
 independent for each $0\leq t \leq T$.
\end{itemize}

If moreover, $0 \in \mathbb{R}^{k}$ is a regular value of the map $J=(J_1,\dots,J_k):V \subseteq \mathbb{R}^{n}\rightarrow \mathbb{R}^{k},$ then for any choice of smooth functions $h_1,\dots, h_p \in\mathcal{C}^{\infty}(V,\mathbb{R})$ such that
$$\int_{0}^{T}h_1 (\gamma(s))\mathrm{d}s<0, \dots, \int_{0}^{T}h_p (\gamma(s))\mathrm{d}s<0,$$
$\Gamma$, as a periodic orbit of the dynamical system $\dot x =X(x)+X_{0}(x)$, $x\in V$,
$$
X_{0}=\left\| \bigwedge_{i=1}^{p} \nabla D_i\wedge\bigwedge_{j=1}^{k} \nabla J_j \right\|_{n-1}^{-2}\cdot\sum_{i=1}^{p}(-1)^{n-i}h_i D_i \Theta_i, 
$$
$$
\Theta_i = \star\left[ \bigwedge_{j=1, j\neq i}^{p} \nabla D_j \wedge \bigwedge_{l=1}^{k} \nabla J_l  \wedge\star\left(\bigwedge_{j=1}^{p} \nabla D_j\wedge\bigwedge_{l=1}^{k} \nabla J_l \right)\right],
$$
is orbitally phase asymptotically stable, with respect to perturbations in $V$, along the invariant manifold $J^{-1}(\{0\})$.

On the other hand, for any choice of smooth functions
$k_1,\dots, k_p \in\mathcal{C}^{\infty}(V,\mathbb{R}),$ such that there exists $i_0 \in \{1,\dots,p\}$ for which $$\int_{0}^{T}k_{i_0} (\gamma(s))\mathrm{d}s>0,$$
$\Gamma$, as a periodic orbit of the dynamical system $\dot x =X(x)+X_{0}(x)$, $x\in V$,
$$
X_{0}=\left\| \bigwedge_{i=1}^{p} \nabla D_i\wedge\bigwedge_{j=1}^{k} \nabla J_j \right\|_{n-1}^{-2}\cdot\sum_{i=1}^{p}(-1)^{n-i}k_i D_i \Theta_i, 
$$
$$
\Theta_i = \star\left[ \bigwedge_{j=1, j\neq i}^{p} \nabla D_j \wedge \bigwedge_{l=1}^{k} \nabla J_l  \wedge\star\left(\bigwedge_{j=1}^{p} \nabla D_j\wedge\bigwedge_{l=1}^{k} \nabla J_l \right)\right],
$$
is an unstable periodic orbit.

Here, $\star$ stands for the Hodge star operator for multi-vector fields, and $\|X_1 \wedge\dots\wedge X_{n-1}\|^{2}_{n-1}$ is equal to the Gram determinant of the set of smooth vector fields $\{X_1,\dots,X_{n-1}\}\subset \mathfrak{X}(U)$.
\end{theorem}

Next theorem gives a perturbation of a completely integrable system such that the resulting system keeps invariant an a-priori fixed periodic orbit and implements a prescribed stability behavior of the orbit, this time with respect to perturbations located in a full open neighborhood of the orbit. The price to be paid is the loss of the conservative nature of the perturbed system. More precisely, the perturbed system will not have any globally defined first integral, since the characteristic multiplier $1$ of the periodic orbit to be stabilized has multiplicity one.

\begin{theorem}\label{NT}
Let $\dot x= X(x)$ be a completely integrable dynamical system generated by a smooth vector field $X\in\mathfrak{X}(U)$ defined eventually on an open subset $U\subseteq \mathbb{R}^{n}$ such that there exists a vector type first integral $I:=(I_1,\dots,I_{n-1})\in\mathcal{C}^{\infty}(U,\mathbb{R}^{n-1})$ of full rank on an open subset $V \subseteq U$. Suppose the system $\dot x= X(x)$ admits a $T-$periodic orbit $\Gamma=\{\gamma(t)\in V : 0\leq t\leq T \}$ such that:
\begin{itemize}
\item $\Gamma\subset I^{-1}(\{0\})$, and $0 \in \mathbb{R}^{n-1}$ is a regular value of the map $I:V\subseteq \mathbb{R}^{n}\rightarrow \mathbb{R}^{n-1}$,
\item $\nabla I_1(\gamma(t)),\dots,\nabla I_{n-1}(\gamma(t)), X(\gamma(t))$ are linearly independent for each $0\leq t \leq T$.
\end{itemize}
Then for any choice of smooth functions $h_1,\dots, h_{n-1} \in\mathcal{C}^{\infty}(V,\mathbb{R})$ such that
$$\int_{0}^{T}h_1 (\gamma(s))\mathrm{d}s<0, \dots, \int_{0}^{T}h_{n-1} (\gamma(s))\mathrm{d}s<0,$$
$\Gamma$, as a periodic orbit of the dynamical system $\dot x =X(x)+\mathbb{X}_{0}(x)$, $x\in V$,
$$
\mathbb{X}_{0}=\left\| \bigwedge_{i=1}^{n-1} \nabla I_i\right\|_{n-1}^{-2}\cdot\sum_{i=1}^{n-1}(-1)^{n-i}h_i I_i \Theta_i, ~where~
\Theta_i = \star\left[ \bigwedge_{j=1, j\neq i}^{n-1} \nabla I_j  \wedge\star\left(\bigwedge_{j=1}^{n-1} \nabla I_j \right)\right],
$$
is orbitally phase asymptotically stable, with respect to perturbations in $V$.

On the other hand, for any choice of smooth functions
$k_1,\dots, k_{n-1} \in\mathcal{C}^{\infty}(V,\mathbb{R}),$ such that there exists $i_0 \in \{1,\dots,n-1\}$ for which $$\int_{0}^{T}k_{i_0} (\gamma(s))\mathrm{d}s>0,$$
$\Gamma$, as a periodic orbit of the dynamical system $\dot x =X(x)+\mathbb{X}_{0}(x)$, $x\in V$,
$$
\mathbb{X}_{0}=\left\| \bigwedge_{i=1}^{n-1} \nabla I_i\right\|_{n-1}^{-2}\cdot\sum_{i=1}^{n-1}(-1)^{n-i}k_i I_i \Theta_i, ~where~
\Theta_i = \star\left[ \bigwedge_{j=1, j\neq i}^{n-1} \nabla I_j  \wedge\star\left(\bigwedge_{j=1}^{n-1} \nabla I_j \right)\right],
$$
is an unstable periodic orbit.
\end{theorem}
\begin{proof}
As the perturbed vector field $Y:=X+\mathbb{X}_0$ verifies by construction that $$\mathcal{L}_{Y}I_j =\mathcal{L}_{X}I_j+\mathcal{L}_{\mathbb{X}_0}I_j=0+h_j I_j=h_j I_j, ~ (\forall)j\in\{1,\dots,n-1\},$$ the rest of the proof is a direct consequence of the following result from \cite{TudoranJGM}.
\end{proof}

\begin{theorem}[\cite{TudoranJGM}]\label{MOST}
Let $\dot x= Y(x)$ be a dynamical system generated by a smooth vector field $Y\in\mathfrak{X}(U)$ defined eventually on an open subset $U\subseteq \mathbb{R}^{n}$, such that there exist $k,p\in\mathbb{N}$, $k+p=n-1$, and respectively $J_1,\dots,J_k, D_1,\dots, D_p$, $h_1$, $\dots$, $h_p \in\mathcal{C}^{\infty}(U,\mathbb{R})$ such that $\mathcal{L}_{Y}J_1 =\dots= \mathcal{L}_{Y}J_k =0$, and $\mathcal{L}_{Y}D_1 =h_1 D_1$, $\dots$, $\mathcal{L}_{Y}D_p =h_p D_p$. Suppose that $\Gamma=\{\gamma(t)\in U : 0\leq t\leq T \}$ is a $T-$periodic orbit of $\dot x= Y(x)$, such that the following conditions hold true:
\begin{itemize}
\item $\Gamma\subset JD^{-1}(\{0\})$, and $0 \in \mathbb{R}^{n-1}$ is a regular value of the map $$JD=(J_1,\dots,J_k,D_1,\dots,D_p):U\subseteq \mathbb{R}^{n}\rightarrow \mathbb{R}^{n-1},$$
\item $\nabla J_1(\gamma(t)),\dots,\nabla J_k(\gamma(t)),\nabla D_1(\gamma(t)),\dots, \nabla D_p (\gamma(t)), Y(\gamma(t))$ are linearly \\
independent for each $0\leq t \leq T$.
\end{itemize}
Then, the characteristic multipliers of the periodic orbit $\Gamma$ are $$\underbrace{1,\dots,1}_{k+1 \ \mathrm{times}}, \exp\left(\int_{0}^{T}h_1 (\gamma(s))ds\right), \dots, \exp\left(\int_{0}^{T}h_p (\gamma(s))ds\right).$$
\end{theorem}

Let us consider now a locally generic three-dimensional Hamiltonian system, i.e., a dynamical system of the type
\begin{align}\label{ghss}
\dfrac{\mathrm{d}u}{\mathrm{d}t} =\nu(u)\left(\nabla H(u)\times \nabla C(u)\right),
\end{align}
where $H,C,\nu \in C^{\infty}(U,\mathbb{R})$ are three given smooth real functions defined on an open subset $U\subseteq \mathbb{R}^3$. 

Recall from the previous section that this system is a Hamiltonian system modeled on the Poisson manifold $(U,\Pi_{\nu, C})$, and moreover, it is completely integrable, since it admits two first integrals, $C$ and $H$. In order to eliminate trivial cases, we suppose that $H$ and $C$ are functionally independent at least on an open subset $V\subseteq U$.

The following result implies a characterization of the set of equilibrium points of the system \eqref{ghss}, generated by the smooth vector field $X(u)=\nu(u)\left(\nabla H(u)\times \nabla C(u)\right)$, for every $u\in U$.
\begin{proposition}\label{prpimp}
The vectors $\nabla H(u), \nabla C(u), X(u)$,  are linearly independent if and only if $u$ is not an equilibrium point of the dynamical system \eqref{ghss}.
\end{proposition}
\begin{proof}
Note first that the vectors $\nabla H(u), \nabla C(u), X(u)$, are linearly independent if and only if $$\det(\nabla H(u)\mid \nabla C(u) \mid X(u))\neq 0.$$ Using the properties of the triple and respectively the canonical scalar product on $\mathbb{R}^3$, one obtains that for every $u\in U$ the following equalities hold true
\begin{align*}
\det(\nabla H(u)\mid\nabla C(u)\mid X(u))&=\det(\nabla H(u) \mid \nabla C(u) \mid \nu (u)(\nabla H(u)\times \nabla C(u)))\\
&=\nu (u) \det(\nabla H(u) \mid \nabla C(u) \mid \nabla H(u)\times \nabla C(u))\\
&=\nu (u)\langle \nabla H(u), [\nabla C(u)\times (\nabla H(u)\times \nabla C(u))]\rangle \\
&=\nu (u)\langle \nabla H(u)\times \nabla C(u), \nabla H(u)\times \nabla C(u)\rangle\\
&=\nu (u) \|\nabla{H}(u)\times \nabla C (u)\|^2.
\end{align*}
Hence, it follows that the vectors $\nabla H(u), \nabla C(u), X(u)$, are linearly independent if and only if $\nu (u)\neq 0$ and $\nabla{H}(u)\times \nabla C (u) \neq \overline{0}$, which is equivalent to $X(u)\neq \overline{0}$, i.e., $u$ is not an equilibrium point of \eqref{ghss}.
\end{proof}

\medskip
An extremely useful consequence of the above result asserts that the second hypothesis of the Theorem \ref{OST} is always verified for locally generic three-dimensional completely integrable systems, i.e., for Hamiltonian systems of the type \eqref{ghss}. 
\medskip

\begin{remark}\label{superrem}
Suppose the dynamical system \eqref{ghss} admits a periodic orbit $\Gamma=\{\gamma(t)\in U : 0\leq t\leq T \}$ with period $T>0$. Then the vectors $\nabla H (\gamma(t)), \nabla C(\gamma(t)), X(\gamma(t))$ are linearly independent for each $0\leq t \leq T$, since periodic orbits contain no equilibrium points.
\end{remark}
Next theorem is the main result of this paper and is an improvement of Theorem \ref{OST} and Theorem \ref{NT}, in the case of locally generic  three-dimensional Hamiltonian dynamical systems. Note that in this case, the second hypothesis of Theorem \ref{OST} (and Theorem \ref{NT}) is always verified and moreover, one needs not know an explicit parameterization of the periodic orbit to be analyzed. All we need to know is the geometric location of the orbit. Since in the case of a locally generic three-dimensional Hamiltonian system \eqref{ghss}, any orbit is located on some common level set of $H$ and respectively $C$, we shall consider two classes of perturbations, namely, one which keeps dynamically invariant the Hamiltonian $H$, and respectively one which keeps dynamically invariant the Casimir function $C$. In both cases, the periodic orbit to be analyzed will remain a periodic orbit for the perturbed system too. 
\begin{theorem}\label{3DOST}
Let 
\begin{equation}\label{EPR}
\dfrac{\mathrm{d}u}{\mathrm{d}t} =\nu(u)\left(\nabla H(u)\times \nabla C(u)\right), \quad u\in U
\end{equation}
be a locally generic three-dimensional Hamiltonian dynamical system, realized as a Hamiltonian system on the Poisson manifold $(U,\Pi_{\nu, C})$, where $U\subseteq \mathbb{R}^3$ is an open set, and $H,C,\nu \in C^{\infty}(U,\mathbb{R})$ are given smooth real functions, such that $H$ and $C$ are functionally independent on an open subset $V\subseteq U$. 

Suppose there exists $\Gamma\subset V$ a periodic orbit of \eqref{EPR}. If $\Gamma \subseteq (H,C)^{-1}(\{(h,c)\})$, where $(h,c)\in\mathbb{R}^2$ is a regular value for the map $(H,C):U\rightarrow \mathbb{R}^2$, then the following conclusions hold true.
\begin{enumerate}
\item If $c$ is a regular value of the map $C:U\rightarrow \mathbb{R}$, then \textbf{for every smooth function} $\alpha \in C^{\infty}(V,(0,\infty))$:
\begin{enumerate}
\item $\Gamma$, as a periodic orbit of the perturbed dynamical system
\begin{align*}
\dfrac{\mathrm{d}u}{\mathrm{d}t} =\nu(u)\left(\nabla H(u)\times \nabla C(u)\right)-\alpha(u)(H(u)-h)\left[\nabla C(u)\times \left(\nabla H(u)\times \nabla C(u)\right)\right],
\end{align*}
$u\in V$, is orbitally phase asymptotically stable, with respect to perturbations in $V$, along the invariant manifold $C^{-1}(\{c\})$.
\item $\Gamma$, as a periodic orbit of the perturbed dynamical system
\begin{align*}
\dfrac{\mathrm{d}u}{\mathrm{d}t} =\nu(u)\left(\nabla H(u)\times \nabla C(u)\right)+\alpha(u)(H(u)-h)\left[\nabla C(u)\times \left(\nabla H(u)\times \nabla C(u)\right)\right],
\end{align*}
$u\in V$, is unstable.
\end{enumerate}
\item On the other hand, \textbf{for every pair of smooth functions} $\alpha,\beta \in C^{\infty}(V,(0,\infty))$:
\begin{enumerate}
\item $\Gamma$, as a periodic orbit of the perturbed dynamical system
\begin{align*}
\dfrac{\mathrm{d}u}{\mathrm{d}t} =\nu(u)\left(\nabla H(u)\times \nabla C(u)\right)&-\alpha(u)(H(u)-h)\left[\nabla C(u)\times \left(\nabla H(u)\times \nabla C(u)\right)\right]\\
& + \beta(u)(C(u)-c)\left[\nabla H(u)\times \left(\nabla H(u)\times \nabla C(u)\right)\right],
\end{align*}
$u\in V$, is orbitally phase asymptotically stable, with respect to perturbations in $V$.
\item $\Gamma$ is unstable, as a periodic orbit of the perturbed dynamical system obtained from the above system by changing the sign in front of $\alpha$ or $\beta$.
\end{enumerate}
\end{enumerate}
\end{theorem}
\begin{proof}
The conclusion follows mainly by Theorem \ref{OST}, Theorem \ref{NT} and Remark \ref{superrem}. In order to apply these results, recall first from the previous section that the system \eqref{EPR} is a completely integrable system on the open subset $U\subseteq \mathbb{R}^3$, since the smooth real functions $H$ and $C$ are two first integrals of \eqref{EPR}, functionally independent on an open subset $V\subseteq U$. 

Recall also from Remark \ref{superrem} that $\nabla H(u), \nabla C(u), X(u):=\nu(u)\left(\nabla H(u)\times \nabla C(u)\right)$ are three linearly independent vectors, for every $u\in\Gamma$, and hence the second hypothesis of Theorem \ref{OST} is verified. 

Since $H$ and $C$ are first integrals, their values are constants along the solutions of \eqref{EPR}, and hence as $\Gamma$ is a periodic orbit, there exists $(h,c)\in\mathbb{R}^2$ such that $(H,C)(u)=(h,c)$, for every $u\in\Gamma$. This is equivalent to $H(u)-h=0$, and $C(u)-c=0$, for every $u\in \Gamma$, and consequently $\Gamma$ will be also a periodic orbit for the above defined perturbed systems, given by items $(1)$ and $(2)$. 

Let us show now that $C^{-1}(\{c\})$ is a dynamically invariant set for the perturbed dynamical system from $(1)$, i.e.,  
\begin{align*}
\dfrac{\mathrm{d}u}{\mathrm{d}t} =\nu(u)\left(\nabla H(u)\times \nabla C(u)\right)-\alpha(u)(H(u)-h)\left[\nabla C(u)\times \left(\nabla H(u)\times \nabla C(u)\right)\right], u\in V.
\end{align*}
This follows if one proves that $C$ is a first integral for the above dynamical system. Indeed, if we denote for every $u\in V$
\begin{equation}\label{vfxz}
X_0 (u):=-\alpha(u)(H(u)-h)\left[\nabla C(u)\times \left(\nabla H(u)\times \nabla C(u)\right)\right],
\end{equation}
then the following equalities prove this assertion.
\begin{align*}
&\left(\mathcal{L}_{X+X_0}C \right)(u)=\langle X(u)+X_0 (u), \nabla C(u)\rangle\\
&=\langle \nu(u)\left(\nabla H(u)\times \nabla C(u)\right)-\alpha(u)(H(u)-h)\left[\nabla C(u)\times \left(\nabla H(u)\times \nabla C(u)\right)\right], \nabla C(u)\rangle\\
&=\nu(u)\langle \nabla H(u)\times \nabla C(u),\nabla C(u) \rangle - \alpha(u)(H(u)-h) \langle \nabla C(u)\times \left(\nabla H(u)\times \nabla C(u)\right),\nabla C(u) \rangle\\
&=0-\alpha(u)(H(u)-h) \langle \nabla C(u)\times \left(\nabla H(u)\times \nabla C(u)\right),\nabla C(u) \rangle\\
&=-\alpha(u)(H(u)-h)\langle \langle\nabla C(u),\nabla C(u)\rangle \nabla H(u) - \langle\nabla C(u),\nabla H(u) \rangle \nabla C(u) ,\nabla C(u) \rangle\\
&=-\alpha(u)(H(u)-h) \left( \|\nabla C(u)\|^2 \langle \nabla H(u),\nabla C(u )\rangle - \langle \nabla C(u),\nabla H(u)\rangle\langle \nabla C(u), \nabla C(u) \rangle \right)\\
&=-\alpha(u)(H(u)-h)\left( \|\nabla C(u)\|^2 \langle \nabla H(u),\nabla C(u )\rangle - \langle \nabla C(u),\nabla H(u)\rangle \|\nabla C(u)\|^2 \right)\\
&=0.
\end{align*}

\medskip
Now we have all ingredients necessary to apply the Theorem \ref{OST} for $J:=C-c$ and $D:=H-h$.
\medskip

$(1a)$ In order to prove this item, we apply the Theorem \ref{OST} for the completely integrable system \eqref{EPR}, a periodic orbit $\Gamma \subseteq (H,C)^{-1}(\{(h,c)\})$, (where $(h,c)$ is a regular value of the map $(H,C)$, and $c$ a regular value for the map $C$), and the first integrals $J,D: V\rightarrow \mathbb{R}$, $J(u)=C(u)-c$, $D(u)=H(u)-h$, where $V\subseteq U$ is an open set where $H$ and $C$ (and implicitly $D$ and $J$) are functionally independent. 

If one denotes by $X$ the vector field which generates the system \eqref{EPR}, then by Theorem \ref{OST}, we get that for any smooth function $\zeta \in C^{\infty}(V,(0,\infty))$, $\Gamma$, as a periodic orbit of the system
\begin{align*}
\dfrac{\mathrm{d}u}{\mathrm{d}t} &=X(u)+X_0 (u),  \quad u\in V,\\
X_{0}(u)&= \|\nabla D(u)\wedge\nabla J(u)\|_{2}^{-2}\cdot (-1)^{3-1}(-\zeta(u))D(u)\star[\nabla J(u)\wedge\star(\nabla D(u)\wedge\nabla J(u))],
\end{align*} 
is orbitally phase asymptotically stable, with respect to perturbations in $V$, along the invariant manifold $J^{-1}(\{0\})=C^{-1}(\{c\})$.

Note that the integral condition from Theorem \ref{OST} is verified because $\operatorname{Im}(-\zeta)\subseteq (-\infty,0)$, and hence one obtains that $\int_{0}^{T}(-\zeta(\gamma(t)))\mathrm{d}t<0$, for any parameterization $t\in[0,T]\mapsto \gamma(t)\in \Gamma$ of the $T$-periodic orbit $\Gamma$.

In order to finish the proof of item $(1)$, it is enough to show that the vector field $X_0$ has the same expression as \eqref{vfxz}.

This follows straightforward taking into account that for any $u,v\in \mathbb{R}^3$ the following equalities hold true:
\begin{equation*}
u\times v = \star(u\wedge v), \quad  \|u\wedge v\|_2 =\|\star(u\wedge v)\| = \|u\times v\|.
\end{equation*}
Indeed, since $\nabla J =\nabla C$ and $\nabla D=\nabla H$, we obtain successively that for any $u\in V$
\begin{align*}
X_{0}(u)&= \|\nabla D(u)\wedge\nabla J(u)\|_{2}^{-2}\cdot (-1)^{3-1}(-\zeta(u))D(u)\star[\nabla J(u)\wedge\star(\nabla D(u)\wedge\nabla J(u))]\\
&=\dfrac{-\zeta (u)}{\|\nabla H(u)\times \nabla C(u)\|^2} (H(u)-h)[\nabla C(u)\times (\nabla H(u)\times \nabla C(u))].
\end{align*}
If one denotes 
\begin{equation*}
\alpha (u):=\dfrac{\zeta (u)}{\|\nabla H(u)\times\nabla C(u)\|^{2}}, \quad u\in V,
\end{equation*}
then the expression of vector filed $X_0$ is the same as \eqref{vfxz}. 

Since $\operatorname{sign}(\alpha)=\operatorname{sign}(\zeta)$, we obtain that $\int_{0}^{T}(-\alpha(\gamma(t)))\mathrm{d}t<0$ for any parameterization $t\in[0,T]\mapsto \gamma(t)\in \Gamma$ of the $T$-periodic orbit $\Gamma$, and consequently as all of the hypothesis of the first part of Theorem \ref{OST} are verified, we get the conclusion.

$(1b)$ This item follows by the same arguments as above, except that we replace in the expression of the vector field $X_0$, the control function $\zeta$ by $-\zeta $, and then apply the second part of Theorem \ref{OST}.

$(2a)/(2b)$ These items follow by applying Theorem \ref{NT} and using the same arguments as those used in order to prove items $(1a)/(1b)$.
\end{proof}

\medskip
\begin{remark}\label{remymp}
If one realize the system \eqref{EPR} as a Hamiltonian system defined on the Poisson manifold $(U,\Pi_{-\nu, H})$, with Hamiltonian $C$, then applying the Theorem \ref{3DOST} to this Hamiltonian realization of the system \eqref{EPR}, we obtain similar results. More precisely, suppose there exists $\Gamma\subset V$ a periodic orbit of \eqref{EPR}. 

If $\Gamma \subseteq (H,C)^{-1}(\{(h,c)\})$, where $(h,c)\in\mathbb{R}^2$ is a regular value for the map $(H,C):U\rightarrow \mathbb{R}^2$, then the following conclusions hold true.

\begin{enumerate}
\item If $h$ is a regular value of the map $H:U\rightarrow \mathbb{R}$, then \textbf{for every smooth function} $\alpha \in C^{\infty}(V,(0,\infty))$:
\begin{enumerate}
\item $\Gamma$, as a periodic orbit of the perturbed dynamical system
\begin{align*}
\dfrac{\mathrm{d}u}{\mathrm{d}t} =\nu(u)\left(\nabla H(u)\times \nabla C(u)\right)+\alpha(u)(C(u)-c)\left[\nabla H(u)\times \left(\nabla H(u)\times \nabla C(u)\right)\right],
\end{align*}
$u\in V$, is orbitally phase asymptotically stable, with respect to perturbations in $V$, along the invariant manifold $H^{-1}(\{h\})$.
\item $\Gamma$, as a periodic orbit of the perturbed dynamical system
\begin{align*}
\dfrac{\mathrm{d}u}{\mathrm{d}t} =\nu(u)\left(\nabla H(u)\times \nabla C(u)\right)-\alpha(u)(C(u)-c)\left[\nabla H(u)\times \left(\nabla H(u)\times \nabla C(u)\right)\right],
\end{align*}
$u\in V$, is unstable.
\end{enumerate}
\item On the other hand, \textbf{for every pair of smooth functions} $\alpha,\beta \in C^{\infty}(V,(0,\infty))$:
\begin{enumerate}
\item $\Gamma$, as a periodic orbit of the perturbed dynamical system
\begin{align*}
\dfrac{\mathrm{d}u}{\mathrm{d}t} =\nu(u)\left(\nabla H(u)\times \nabla C(u)\right)&-\alpha(u)(H(u)-h)\left[\nabla C(u)\times \left(\nabla H(u)\times \nabla C(u)\right)\right]\\
& + \beta(u)(C(u)-c)\left[\nabla H(u)\times \left(\nabla H(u)\times \nabla C(u)\right)\right],
\end{align*}
$u\in V$, is orbitally phase asymptotically stable, with respect to perturbations in $V$.
\item $\Gamma$ is unstable, as a periodic orbit of the perturbed dynamical system obtained from the above system by changing the sign in front of $\alpha$ or $\beta$.
\end{enumerate}
\end{enumerate}
\end{remark}

Next result follows directly from Theorem \ref{3DOST} and provides a method to stabilize asymptotically with phase an arbitrary fixed periodic orbit of a two-dimensional Hamiltonian dynamical system. Before stating the stabilization method, let us recall that a \textit{two-dimensional Hamiltonian dynamical system} generated by a Hamiltonian $\mathcal{H}\in\mathcal{C}^{\infty}(\Omega,\mathbb{R})$ defined on an open set $\Omega\subseteq\mathbb{R}^2$, is given by
$$
\dfrac{\mathrm{d}v}{\mathrm{d}t}=\mu(v)\mathbb{J}\nabla \mathcal{H}(v), ~ v\in \Omega,
$$ 
where $\mu\in\mathcal{C}^{\infty}(\Omega,\mathbb{R})$ and $\mathbb{J}= \left[ 
\begin{array}{cc}
   0 & 1 \\
   -1 & 0    
\end{array} \right].$
Note that any two-dimensional Hamiltonian dynamical systems is completely integrable, since the Hamiltonian $\mathcal{H}$ is a first integral, and conversely, any completely integrable two-dimensional system is locally a two-dimensional Hamiltonian dynamical system. 

Let us state now the two-dimensional version of the stabilization method given in Theorem \ref{3DOST}.

\begin{proposition}\label{2DHH}
Let 
\begin{equation}\label{2H}
\dfrac{\mathrm{d}v}{\mathrm{d}t}=\mu(v)\mathbb{J}\nabla \mathcal{H}(v), ~ v\in \Omega,
\end{equation}
be a two-dimensional Hamiltonian dynamical system, where $\Omega\subseteq\mathbb{R}^2$ is an open set and  $\mu,\mathcal{H}\in\mathcal{C}^{\infty}(\Omega,\mathbb{R})$ are given smooth real functions such that $\mathcal{H}$ do not have any critical point in some open subset $W\subseteq \Omega$.

Suppose there exists $\Gamma\subset W$ a periodic orbit of \eqref{2H}. If $\Gamma \subseteq \mathcal{H}^{-1}(\{h\})$ where $h\in\mathbb{R}$ is a regular value of $\mathcal{H}$, then \textbf{for every smooth function} $\alpha \in C^{\infty}(W,(0,\infty))$:
\begin{enumerate}
\item $\Gamma$, as a periodic orbit of the perturbed dynamical system
\begin{align*}
\dfrac{\mathrm{d}v}{\mathrm{d}t} =\mu(v)\mathbb{J}\nabla \mathcal{H}(v) - \alpha(v)(\mathcal{H}(v)-h)\nabla \mathcal{H}(v), ~ v\in W,
\end{align*}
is orbitally phase asymptotically stable, with respect to perturbations in $W$.
\item $\Gamma$, as a periodic orbit of the perturbed dynamical system
\begin{align*}
\dfrac{\mathrm{d}v}{\mathrm{d}t} =\mu(v)\mathbb{J}\nabla \mathcal{H}(v) + \alpha(v)(\mathcal{H}(v)-h)\nabla \mathcal{H}(v), ~ v\in W,
\end{align*}
is unstable.
\end{enumerate}
\end{proposition}
\begin{proof}
Using the notations of Theorem \ref{3DOST}, we define $U:=\Omega\times (-\varepsilon,\varepsilon)$, $V:=W \times (-\varepsilon,\varepsilon)$ (where $\varepsilon >0$ is a fixed real number), and $H(v,z):=\mathcal{H}(v)$, $C(v,z):=z$, $\nu(v,z):=\mu(v)$ for every $(v,z)\in U$. The rest of the proof follows from Theorem \ref{3DOST} applied to the periodic orbit $\Gamma\times\{0\}$, and perturbations along $W\times\{0\}$. 
\end{proof}

\bigskip
Let us now illustrate the stabilization result introduced in Theorem \ref{3DOST} in the case of a Hamiltonian version of the Rikitake system.

\begin{example}
Let us start by recalling that the Rikitake system is a dynamical system which provides a mathematical model for the irregular polarity switching of Earth's magnetic field \cite{rikitake}. Since this system exhibits a chaotic behavior, a good knowledge of the conservative part, leads to better understanding of the provenience of its complex behavior. Let us consider now the conservative part of the Rikitake system studied in \cite{siamt}, namely the dynamical system generated by the vector field
$$
X(x,y,z)=\left( yz+\beta y \right) \dfrac{\partial}{\partial x}+ \left( xz - \beta x \right) \dfrac{\partial}{\partial y} -xy \dfrac{\partial}{\partial z}\in\mathfrak{X}(\mathbb{R}^3),
$$
where $\beta$ is a real parameter.

As proved in \cite{siamt}, the induced dynamical system, \begin{equation}\label{HOOG}\dot{u}=X(u), \ u=(x,y,z)\in\mathbb{R}^{3},\end{equation} admits a three-dimensional Hamiltonian realization of the type \eqref{EPR}, where
$$
\nu(x,y,z)=1, \ H(x,y,z)=\dfrac{1}{4}\left( -x^2 +y^2 \right)-\beta z, \ C(x,y,z)=\dfrac{1}{2}\left(x^2 + y^2 \right) + z^2.
$$

Note that the maximal set where $\nabla H$ and $\nabla C$ are linearly independent, is the open set given by the complement of the set of equilibrium points of \eqref{HOOG}, namely $$V:=\mathbb{R}^3 \setminus\{\{(x,0,\beta): x\in\mathbb{R}\}\cup \{(0,y,-\beta): y\in\mathbb{R}\} \cup \{(0,0,z): z\in\mathbb{R}\}\}.$$
 
Recall from \cite{siamt} that there exists an open and dense subset of the image of the map $(H,C):\mathbb{R}^3 \rightarrow \mathbb{R}^2$, such that each fiber of any element $(h,c)$ from this set, corresponds to periodic orbits of Euler's equations. Moreover, any such element is a regular value of $(H,C)$, as well as its components for the corresponding maps, $H$ and respectively $C$.

Let $(h,c)\in \mathbb{R}^2$ belongs to the above mention set, and let $\Gamma\subseteq (H,C)^{-1}(\{(h,c)\})$ be a periodic orbit of the dynamical system \eqref{HOOG}.

Then by Theorem \ref{3DOST} and the Remark \ref{remymp}, the following conclusions hold true.
\begin{enumerate}
\item \textbf{For every smooth function} $a \in C^{\infty}(V,(0,\infty))$:
\begin{enumerate}
\item $\Gamma$, as a periodic orbit of the perturbed dynamical system
\begin{equation*}
\left\{\begin{array}{lll}
\dot{x}=& yz + \beta y -a(x,y,z)\left[\dfrac{1}{4}\left( -x^2 +y^2 \right)-\beta z-h\right] x\left( -y^2 -2z^2+2\beta z \right)\\
\dot{y}=& xz - \beta x -a(x,y,z)\left[\dfrac{1}{4}\left( -x^2 +y^2 \right)-\beta z-h\right] y\left( x^2 +2z^2 +2\beta z \right)\\
\dot{z}=& -xy            -a(x,y,z)\left[\dfrac{1}{4}\left( -x^2 +y^2 \right)-\beta z-h\right]\left[ z(x^2 - y^2)-\beta (x^2 + y^2)\right]
\end{array}
\right.
\end{equation*}
$u\in V$, is orbitally phase asymptotically stable, with respect to perturbations in $V$, along the invariant manifold 
\begin{equation*}
C^{-1}(\{c\})=\left\{(x,y,z)\in\mathbb{R}^3 \mid \dfrac{1}{2}\left(x^2 + y^2 \right) + z^2=c \right\}.
\end{equation*}
\item $\Gamma$ is unstable, as a periodic orbit of the perturbed dynamical system obtained from the above system by changing the sign in front of $a$.
\end{enumerate}

\item \textbf{For every smooth function} $a\in C^{\infty}(V,(0,\infty))$:
\begin{enumerate}
\item $\Gamma$, as a periodic orbit of the perturbed dynamical system
\begin{equation*}
\left\{\begin{array}{lll}
\dot{x}=& yz + \beta y +a(x,y,z)\left[\dfrac{1}{2}\left(x^2 + y^2 \right) + z^2-c\right] x\left( -\dfrac{y^2}{2}+\beta z -{\beta}^2 \right)\\
\dot{y}=& xz - \beta x +a(x,y,z)\left[\dfrac{1}{2}\left(x^2 + y^2 \right) + z^2-c\right] y\left( -\dfrac{x^2}{2}-\beta z -{\beta}^2 \right)\\
\dot{z}=& - xy           +a(x,y,z)\left[\dfrac{1}{2}\left(x^2 + y^2 \right) + z^2-c\right]\dfrac{1}{2}\left[ -z(x^2 +y^2)+\beta (x^2- y^2)\right]
\end{array}
\right.
\end{equation*}
$u\in V$, is orbitally phase asymptotically stable, with respect to perturbations in $V$, along the invariant manifold 
\begin{equation*}
H^{-1}(\{h\})=\left\{(x,y,z)\in\mathbb{R}^3 \mid \dfrac{1}{4}\left( -x^2 +y^2 \right)-\beta z=h \right\}.
\end{equation*}
\item $\Gamma$ is unstable, as a periodic orbit of the perturbed dynamical system obtained from the above system by changing the sign in front of $a$.
\end{enumerate}

\item \textbf{For every pair of smooth functions} $a,b \in C^{\infty}(V,(0,\infty))$:
\begin{enumerate}
\item $\Gamma$, as a periodic orbit of the perturbed dynamical system
\begin{alignat*}{2}
\left\{\begin{array}{lll}
\dot{x}= yz + \beta y - a(x,y,z)\left[\dfrac{1}{4}\left( -x^2 +y^2 \right)-\beta z-h\right] x\left( -y^2 -2z^2+2\beta z \right)\\
+ b(x,y,z)\left[\dfrac{1}{2}\left(x^2 + y^2 \right) + z^2-c\right] x\left( -\dfrac{y^2}{2}+\beta z -{\beta}^2 \right)\\
\dot{y}= xz - \beta x -a(x,y,z)\left[\dfrac{1}{4}\left( -x^2 +y^2 \right)-\beta z-h\right] y\left( x^2 +2z^2 +2\beta z \right)\\
+b(x,y,z)\left[\dfrac{1}{2}\left(x^2 + y^2 \right) + z^2-c\right] y\left( -\dfrac{x^2}{2}-\beta z -{\beta}^2 \right)\\
\dot{z}= -xy            -a(x,y,z)\left[\dfrac{1}{4}\left( -x^2 +y^2 \right)-\beta z-h\right]\left[ z(x^2 - y^2)-\beta (x^2 + y^2)\right]\\
+b(x,y,z)\left[\dfrac{1}{2}\left(x^2 + y^2 \right) + z^2-c\right]\dfrac{1}{2}\left[ -z(x^2 +y^2)+\beta (x^2- y^2)\right]
\end{array}
\right.
\end{alignat*}
$u\in V$, is orbitally phase asymptotically stable, with respect to perturbations in $V$.
\item $\Gamma$ is unstable, as a periodic orbit of the perturbed dynamical system obtained from the above system by changing the sign in front of $a$ or $b$.
\end{enumerate}

\end{enumerate}
\end{example}

\bigskip
\bigskip

\noindent {\sc R.M. Tudoran}\\
West University of Timi\c soara\\
Faculty of Mathematics and Computer Science\\
Department of Mathematics\\
Blvd. Vasile P\^arvan, No. 4\\
300223 - Timi\c soara, Rom\^ania.\\
E-mail: {\sf tudoran@math.uvt.ro}\\
\medskip

\end{document}